\pdfoutput=1

\documentclass[10pt]{article}

\makeatletter
\renewcommand{\maketitle}{\bgroup\setlength{\parindent}{0pt}
\begin{flushleft}
  \textbf{\@title}

  \@author
\end{flushleft}\egroup
} 
\def\footnoterule{\kern-3\p@
  \hrule \@width 2in \kern 2.6\p@} 
\makeatother

\usepackage{preamble}
\usepackage{courier}
\usepackage[square,numbers]{natbib}

\title{{\large \textbf{Product-free sets in approximate subgroups of distal groups}}}
\author{Atticus Stonestrom\ \ \ \ \ \ \ \ \ \ \ \ \ \ \ \ \ \  {\footnotesize (Email: \texttt{atticusstonestrom@yahoo.com})}}
\date{}

\begin{document}
\maketitle

{\small \noindent\textbf{Abstract:} Recall that a subset $X$ of a group $G$ is `product-free' if $X^2\cap X=\varnothing$, ie if $xy\notin X$ for all $x,y\in X$. Let $G$ be a group definable in a distal structure. We prove there are constants $c>0$ and $\delta\in(0,1)$ such that every finite subset $X\subseteq G$ distinct from $\{1\}$ contains a product-free subset of size at least $\delta|X|^{c+1}/|X^2|^c$. In particular, every finite $k$-approximate subgroup of $G$ distinct from $\{1\}$ contains a product-free subset of density at least $\delta/k^c$.

The proof is short, and follows quickly from Ruzsa calculus and an iterated application of Chernikov and Starchenko's distal regularity lemma.\newline

\noindent\textbf{Acknowledgements:} I would like to thank Artem Chernikov, Ben Green, Zach Hunter, and Sergei Starchenko for helpful discussion, and I would especially like to thank Jonathan Pila and Udi Hrushovski for much encouragement and feedback. \newline 

\noindent\textbf{Notation:} Throughout we will use the standard terminology and notation of `Ruzsa calculus'; thus a `multiplicative set' is simply a finite subset of a group. Given multiplicative sets $X,Y$ in the same ambient group, we have $XY=\{xy:x\in X,y\in Y\}$ and $X^{-1}=\{x^{-1}:x\in X\}$ and, for any natural number $n$, $X^n=\{x_1\cdots x_n:x_i\in X\}$. Also, we take $\log$ to mean the logarithm base $2$.}

\section{Introduction}
Product-free sets – especially in the abelian context, where they are called `sum-free' – are a rich topic of study in additive combinatorics; one of the earliest results is Erd\H{o}s' 1965 proof that every finite set of non-zero integers contains a sum-free subset of density at least $1/3$. Given a multiplicative set $X$, two natural questions of interest include the following: what is the largest density of a product-free subset of $X$, and how many product-free subsets does $X$ contain? There is an extensive literature on these and related topics, and we refer the reader to the survey \citep{tao_vu_2017} for a picture of some of this landscape in the abelian case. The starting point of our motivation here is a question raised by Babai and Sós in \citep{babai_sos_1985}: is there a constant $\delta>0$ such that every finite non-trivial group contains a product-free subset of density at least $\delta$? 

In \citep{gowers_2008}, Gowers gave a negative answer to this question. Indeed, let $H$ be a non-trivial finite group and $d$ the minimal dimension of a non-trivial complex representation of $H$. Gowers then shows that $H$ does not contain any product-free set of size greater than $|H|/d^{1/3}$; by a classical result of Frobenius, this is enough to conclude that, for every prime power $q$, $\mathrm{PSL}_2(q)$ contains no product-free set of size greater than $2|\mathrm{PSL}_2(q)|^{8/9}$.

Gowers also obtained a bound in the reverse direction, showing that $H$ \textit{does} contain a product-free set of size at least $|H|/2000^d$. In \citep{nikolov_pyber_2011}, Nikolov and Pyber improved the exponential dependence on $d$ to a polynomial one, showing the existence of an absolute constant $\delta\in(0,1)$ such that $H$ contains a product-free set of size at least $\delta|H|/d$.

These results (and others – see eg \citep{kedlaya_2009} for a survey, and see \citep{palacin} for pseudo-finite analogues due to Palacín) give perspective on the maximal densities of product-free subsets of finite groups. In this paper we are motivated by the natural generalization of this line of inquiry to the class of finite \textit{approximate} groups. We show that, for fixed $k$, a positive answer to the analogue of Babai and Sós' question for $k$-approximate groups can be obtained in a certain model-theoretic context: namely, in the context of a group interpretable in a \textit{distal structure}. Specifically, given such a group $G$, we show the existence of constants $c>0$ and $\delta\in(0,1)$ such that every finite subset $X\subseteq G$ distinct from $\{1\}$ contains a product-free subset of density at least $\delta/k^c$, where $k=|X^2|/|X|$. In the case $G=\mathrm{GL}_n(\mathbb{C})$, we note that, modulo a weighted version of an argument of Alon and Kleitman, this also follows from Breuillard, Green, and Tao's classification of the approximate subgroups of $\mathrm{GL}_n(\mathbb{C})$ (again with polynomial dependence on $k$).

\section{Preliminaries}
\subsection{Approximate groups}Recall that the `doubling' and `tripling' constants of a multiplicative set $X$ are the ratios $|X^2|/|X|$ and $|X^3|/|X|$, respectively. For $k\geqslant 1$, $X$ is said to be a $k$-approximate group if $1\in X=X^{-1}$ and $X^2$ is covered by at most $k$-many translates of $X$; note that the doubling constant of a $k$-approximate group is at most $k$. There is a vast literature on $k$-approximate groups, and we refer the reader to the survey \citep{breuillard_2013} for context on the area. For our main result, we will need only two basic facts. The first is a consequence of Ruzsa's `triangle inequality':
\begin{fact}If $X$ is a multiplicative set with $|X^2|\leqslant k|X|$, then $|XX^{-1}|\leqslant k^2|X|$.
\end{fact}The second is a consequence of a theorem of Petridis; see Theorem 1.5 in \citep{petridis_2012}:
\begin{fact}
If $X$ is a multiplicative set with $|X^2|\leqslant k|X|$, then there is a subset $Y\subseteq X$ such that $|Y|\geqslant |X|/k$ and $|Y^3|\leqslant k^3|Y|$.
\end{fact} In some follow-up discussion, it will also be convenient to cite the following results; the first is a consequence of Theorem 4.6 in \citep{tao_2008}.

\begin{fact}
    Let $Y,Z$ be multiplicative sets of size $n$, and suppose $|YZ|\leqslant kn$. Then, in the ambient group, there is an $O(k^{O(1)})$-approximate group $W$ and an element $u$ such that $|W|\leqslant O(k^{O(1)})n$ and $|Y\cap uW|\geqslant |Y|/O(k^{O(1)})$.
\end{fact}

The second is a consequence of the celebrated Balog-Szemerédi-Gowers theorem; see for example Lemma 5.1 in \citep{tao_2008}.

\begin{fact}
    Let $X$ be a multiplicative set, and suppose there are at least $|X|^2/k$ many pairs $(x,y)\in X\times X$ satisfying $xy\in X$. Then there are subsets $Y,Z\subseteq X$,  each of size at least $|X|/O(k^{O(1)})$, such that $|YZ|\leqslant O(k^{O(1)})|X|$.
\end{fact}

\subsection{Distal structures}
The model-theoretic context for this paper is the notion of distality, a property of first-order structures introduced by Simon in \citep{simon_2013}. We will not need the definition of a distal structure here, and refer the reader to the introduction of \citep{chernikov_starchenko_2018} for an overview. Instead, we point out two motivating examples: (i) the field $(\mathbb{R},\cdot,+)$, or more generally any o-minimal structure, and (ii) the field $(\mathbb{Q}_p,\cdot,+)$ for any prime $p$.

In this note we are concerned with groups $(G,\cdot)$ that are definable in some distal structure. One non-trivial fact is that these are precisely the groups \textit{interpretable} in a distal structure; see \citep{krupinski_portillo} or \citep{aschenbrenner_chernikov_gehret_ziegler} for detailed proofs of this. (As an aside, note that it also follows directly from Fact 2.7 below alone that every relation interpretable in a distal structure has the strong Erd\H{o}s-Hajnal property.)

As a concrete example, the reader may throughout consider $G$ to be the $\mathbb{R}$-points of any semialgebraic group.

The main tool we employ comes from the paper \citep{chernikov_starchenko_2018}, wherein Chernikov and Starchenko showed that relations definable in distal structures enjoy strong combinatorial regularity properties. To state their theorem, we first recall the notion of a `homogeneous' tuple of sets in an $n$-partite hypergraph; let $R\subseteq X_1\times\dots\times X_n$ be an $n$-ary relation on sets $X_i$.

\begin{definition}A tuple $(U_1,\dots,U_n)$ of subsets $U_i\subseteq X_i$ is `$R$-homogeneous' if $U_1\times\dots\times U_n$ is either a subset of $R$ or disjoint from $R$.
\end{definition}
\begin{definition}
$R$ has the `strong Erd\H{o}s-Hajnal property' if there exists a constant $\delta\in(0,1)$ with the following property: for any finite subsets $W_i\subseteq X_i$, there are subsets $U_i\subseteq W_i$, with each $U_i$ of size at least $\delta |W_i|$, such that $(U_1,\dots,U_n)$ is $R$-homogeneous.
\end{definition}

The relevant result from \citep{chernikov_starchenko_2018} is then the following, which is an immediate consequence of Corollary 4.5 there:
\begin{fact}
Every relation definable in a distal structure has the strong Erd\H{o}s-Hajnal property.
\end{fact}

For some examples of the varied interactions between distality and additive combinatorics, we refer the reader to the papers \citep{anderson}, \citep{chernikov_galvin_starchnkoe}, and \citep{conant_pillay_terry}.

\section{Main Result}We now have all the necessary prerequisites. Our result holds in any group in which the $6$-ary relation $R(\overline{x},\overline{y})\equiv x_1x_2x_3=y_1y_2y_3$ has the strong Erd\H{o}s-Hajnal property; by Fact 2.7, it applies in particular to any group definable in a distal structure. For the rest of this section thus fix a group $G$ in which $R$ has the strong Erd\H{o}s-Hajnal property.

\begin{lemma}There is a constant $c_0>0$ such that, for any $\alpha\in(0,1)$, there is some $\varepsilon\in(0,1)$ with the following property: every finite subset $Y\subseteq G$ of size at least $8/\alpha$ contains subsets $U,V,W\subseteq Y$, all of size at least $\varepsilon|Y|^{c_0+1}/|Y^3|^{c_0}$, such that $|UVW|\leqslant\alpha |Y|$.
\end{lemma}
\begin{proof}
Let $\delta\in(0,1)$ be a constant witnessing the strong Erd\H{o}s-Hajnal property for $R$; ie any finite subsets $U_1,U_2,U_3,V_1,V_2,V_3\subseteq G$ contain respective subsets $   U'_1,U'_2,U'_3,V'_1,V'_2,V'_3$, all of density at least $\delta$, such that $(U'_1,\dots,V'_3)$ is $R$-homogeneous. I claim we may take $c_0=\log(1/\delta)$; to see this, fix $\alpha\in(0,1)$ and let $\varepsilon=\delta\alpha^{c_0}$.

First note that if $U_1,U_2,U_3,V_1,V_2,V_3$ are subsets of $G$ of size at least $2$, then we cannot have $u_1u_2u_3=v_1v_2v_3$ for all $u_i\in U_i$ and $v_i\in V_i$. Thus if $(U_1,\dots,V_3)$ is $R$-homogeneous, then in fact $U_1U_2U_3$ and $V_1V_2V_3$ are disjoint, and so in particular one of $U_1U_2U_3$ and $V_1V_2V_3$ has size at most $|W_1W_2W_3|/2$ for any sets $W_i\supseteq U_i\cup V_i$. The strong Erd\H{o}s-Hajnal property hence tells us that any finite subsets $W_1,W_2,W_3\subseteq G$, all of size at least $2$, contain respective subsets $U_1,U_2,U_3$ with $|U_i|\geqslant\delta|W_i|$ for each $i$ and with $|U_1U_2U_3|\leqslant |W_1W_2W_3|/2$.

Now we simply iterate this fact; fix a finite subset $Y\subseteq G$ with $|Y|\geqslant 8/\alpha$ and let $U_1=V_1=W_1=Y$. Applying the previous remark, we inductively define sets $U_n,V_n,W_n$ such that (i) $U_{n+1},V_{n+1},W_{n+1}$ are contained in $U_n,V_n,W_n$, respectively, (ii) $U_{n+1},V_{n+1},W_{n+1}$ have size at least $\delta|U_n|,\delta|V_n|,\delta|W_n|$, respectively, and (iii) $|U_{n+1}V_{n+1}W_{n+1}|\leqslant|U_nV_nW_n|/2$; the $(n+1)$-th tuple of sets in this sequence can be constructed as long as all of $U_n,V_n,W_n$ have size at least $2$, and hence as long as $\delta^{n-1}|Y|\geqslant 2$.

Now let $n=\lceil \log(|Y^3|/\alpha|Y|)\rceil+1$. Then $$\delta^{n-2}\geqslant \delta^{\log(|Y^3|/\alpha|Y|)}=(\alpha|Y|/|Y^3|)^{\log(1/\delta)}=(\alpha|Y|/|Y^3|)^{c_0},$$ whence $\delta^{n-1}|Y|\geqslant\varepsilon|Y|^{c_0+1}/|Y^3|^{c_0}$, and so it suffices to find $U,V,W\subseteq Y$ of size at least $\delta^{n-1}|Y|$ and with $|UVW|\leqslant\alpha|Y|$. If $\delta^{n-2}|Y|\geqslant 2$, then the sets $U_n,V_n,W_n$ exist and are of appropriate size, and further satisfy $|U_nV_nW_n|\leqslant |Y^3|/2^{n-1}\leqslant\alpha|Y|$, as needed. If instead $2>\delta^{n-2}|Y|$, then in particular $2>\delta^{n-1}|Y|$, and so we may take $U,V,W$ to be any subsets of $Y$ of size at least $2$; these will satisfy $|UVW|\leqslant 8\leqslant \alpha|Y|$, again as needed.
\end{proof}

\begin{corollary}
There are constants $c_1>0$ and $\varepsilon_1\in(0,1)$ such that every finite subset $Y\subseteq G$ of size at least $16$ contains a subset $Z\subseteq Y$, of size at least $\varepsilon_1|Y|^{c_1+1}/|Y^3|^{c_1}$, such that $|Z^{-1}ZZ^{-1}|\leqslant|Y|/2$.
\end{corollary}
\begin{proof}
Let $c_0$ be given by Lemma 3.1, and let $\varepsilon_0\in(0,1)$ witness the lemma for the case $\alpha=1/2$; I claim we may take $c_1=3c_0$ and $\varepsilon_1=4\varepsilon_0^3$. To see this, fix a finite subset $Y\subseteq G$ of size at least $16$, and for notational convenience let $k=|Y^3|/|Y|$. By Lemma 3.1 we can find subsets $U,V,W\subseteq Y$, each of size at least $\varepsilon_0|Y|/k^{c_0}$, such that $|UVW|\leqslant|Y|/2$.

Now, for each $g,h\in G$, let $Z_{g,h}=U^{-1}g\cap V\cap h W^{-1}$. The map $(g,h,z)\mapsto (gz^{-1},z,z^{-1}h)$ gives a bijection from $\{(g,h,z):g,h\in G,z\in Z_{g,h}\}$ to $U\times V\times W$, with inverse given by $(u,v,w)\mapsto (uv,vw,v)$, whence $\sum_{g,h\in G}|Z_{g,h}|=|U||V||W|$. On the other hand, $Z_{g,h}\neq\varnothing$ only if $g\in UV$ and $h\in VW$. The size of these sets is bounded above by $|UVW|\leqslant |Y|/2$, so we have $|U||V||W|\leqslant (|Y|^2/4)\sup_{g,h\in G}|Z_{g,h}|$, and there thus exist some $g,h\in G$ such that $$|Z_{g,h}|\geqslant 4|U||V||W|/|Y|^2\geqslant 4\varepsilon_0^3|Y|/k^{3c_0}=\varepsilon_1|Y|^{c_1+1}/|Y^3|^{c_1}.$$ On the other hand, $Z_{g,h}^{-1}Z_{g,h}Z_{g,h}^{-1}\subseteq g^{-1}UVW h^{-1}$, whence $|Z_{g,h}^{-1}Z_{g,h}Z_{g,h}^{-1}|\leqslant |UVW|\leqslant |Y|/2$. So taking $Z=Z_{g,h}$ gives the desired result.
\end{proof}

\begin{theorem}
There are constants $c_2>0$ and $\varepsilon_2\in (0,1)$ such that the following holds: if $X\subseteq G$ is finite and distinct from $\{1\}$, and $k=|X^2|/|X|$, then $X$ contains a product-free subset of size at least $\varepsilon_2|X|/k^{c_2}$.
\end{theorem}
\begin{proof}
Let $c_1$ and $\varepsilon_1$ be given by Corollary 3.2; I claim we may take $c_2=3c_1+4$ and $\varepsilon_2=\min\{\varepsilon_1/2,1/16\}$. Thus fix a finite subset $X\subseteq G$ distinct from $\{1\}$. If $|X|<16k$, then $\varepsilon_2|X|/k^{c_2}< 16k/16k^{c_2}\leqslant 1$, so we may pick any $z\in X\setminus\{1\}$ and then $\{z\}$ will be a product-free subset of $X$ of appropriate size. Hence we may assume $|X|\geqslant 16k$. By Fact 2.2, there is a subset $Y\subseteq X$ with $|Y|\geqslant |X|/k$, hence $|Y|\geqslant 16$, and $|Y^3|\leqslant k^3|Y|$. By Corollary 3.2, there is then a subset $Z\subseteq Y$ with $|Z|\geqslant \varepsilon_1|Y|/k^{3c_1}$ and $|Z^{-1}ZZ^{-1}|\leqslant |Y|/2$.

For each $g\in G$, let $Y_g=gZ\cap Y$. The map $(g,y)\mapsto (y,g^{-1}y)$ gives a bijection from $\{(g,y):g\in G,y\in Y_g\}$ to $Y\times Z$, whence $\sum_{g\in G}|Y_g|=|Y||Z|$. Moreover, $|Y_g|\leqslant|gZ|=|Z|$ for all $g\in G$, and $|Z^{-1}ZZ^{-1}|\leqslant |Y|/2$, so $\sum_{g\in Z^{-1}ZZ^{-1}}|Y_g|\leqslant |Y||Z|/2$; combining these bounds gives $\sum_{g\notin Z^{-1}ZZ^{-1}}|Y_g|\geqslant |Y||Z|/2$.

On the other hand, $Y_g\neq\varnothing$ only if $g\in YZ^{-1}$. Since $Y,Z\subseteq X$ and $|X^2|\leqslant k|X|$, by Fact 2.1 we have $|YZ^{-1}|\leqslant |XX^{-1}|\leqslant k^2|X|\leqslant k^3|Y|$. So in particular $$|Y||Z|/2\leqslant k^3|Y|\sup_{g\notin Z^{-1}ZZ^{-1}}|Y_g|,$$ and there thus exists some $g\in G\setminus Z^{-1}ZZ^{-1}$ such that $|Y_g|\geqslant |Z|/2k^3$.

By definition, $Y_g\subseteq Y\subseteq X$. We further have $|Y_g|\geqslant \varepsilon_1|Y|/2k^{3c_1+3}$ since $|Z|\geqslant \varepsilon_1|Y|/k^{3c_1}$, and therefore $$|Y_g|\geqslant \varepsilon_1|X|/2k^{3c_1+4}\geqslant\varepsilon_2|X|/k^{c_2}$$ since $|Y|\geqslant |X|/k$. So we need only show that $Y_g$ is product-free and we will be done; for this, since $Y_g\subseteq gZ$, it suffices to show that $gZ$ is product-free. Thus suppose otherwise; then there are $z_1,z_2,z_3\in Z$ such that $gz_1gz_2=gz_3$. Rearranging this gives $g=z_1^{-1}z_3z_2^{-1}$, contradicting that $g\notin Z^{-1}ZZ^{-1}$, and so we are done.
\end{proof}

This concludes the main result. For completeness, we would like to note an alternative proof of Theorem 3.3 in the case that $G=\mathrm{GL}_n(\mathbb{C})$. The $k$-approximate subgroups of $G$ were classified qualitatively by Hrushovski in \citep{hrushovski}, and then with polynomial dependence on $k$ by Breuillard, Green, and Tao in \citep{breuillard_green_tao}; the following is an immediate consequence of Theorem 2.5 in \citep{breuillard_green_tao}.

\begin{fact}
    For every $n$, there are constants $c>0$ and $\varepsilon\in(0,1)$ such that the following holds: for any finite $k$-approximate group $X\subseteq\mathrm{GL}_n(\mathbb{C})$, there is a nilpotent group $H\leqslant\mathrm{GL}_n(\mathbb{C})$ of step at most $n-1$ and an element $g\in \mathrm{GL}_n(\mathbb{C})$ such that $|X\cap gH|\geqslant\varepsilon|X|/k^c$.
\end{fact} By Fact 2.3 and Corollary A.3, and noting that any non-trivial coset of a subgroup is product-free, Fact 3.4 implies Theorem 3.3 in the case $G=\mathrm{GL}_n(\mathbb{C})$, with the same polynomial dependence on $k$. It is tempting to ask whether the dependence on $k$ can be removed entirely, either in this case or in the more general distal context:
\begin{question}
Suppose $G$ is a group definable in a distal structure. Is there a constant $\delta>0$ such that every finite subset $X\subseteq G$ distinct from $\{1\}$ contains a product-free subset of size at least $\delta|X|$? How about in the case $G=\mathrm{GL}_n(\mathbb{C})$?
\end{question}For completeness we include some remarks on this question in the next section. 

\section{Remarks on Question 3.5}
In this section we will briefly discuss Question 3.5; these are standard observations but perhaps worth noting nonetheless. Let $G$ be a group definable in a distal structure and let $X\subseteq G$ be a finite subset distinct from $\{1\}$; we wish to find a product-free subset of $X$ of fixed positive density. On the one hand, if there are `very few' pairs $(x,y)\in X\times X$ with $xy\in X$, say at most $k|X|$-many for some fixed $k\geqslant 1$, then one can check that $X$ contains a product-free subset of density at least $1/O(k)$; indeed this is true without any hypotheses on $G$ whatsoever. 

On the other hand, if `most' of the pairs $(x,y)\in X\times X$ satisfy $xy\in X$, then we can couple Theorem 3.3 with the Balog-Szemerédi-Gowers theorem to obtain the desired set. Specifically, if there are at least $|X|^2/k$ such pairs, then $X$ contains a product-free subset of density at least $1/O(k^{O(1)})$; let us quickly sketch the proof of this.

The first relevant observation is that every dense subset of a `coset' of an approximate group contains a dense product-free subset. More precisely, suppose that $W\subseteq G$ is a finite $k$-approximate group, that $u\in G$ is arbitrary, and that $V$ is a subset of $uW$ distinct from $\{1\}$ and of size $|W|/O(k^{O(1)})$. We have two cases; if $u\notin W^{-1}WW^{-1}$, then $uW$ and hence $V$ as well are already product-free, so there is nothing to show. If instead $u$ does lie in $W^{-1}WW^{-1}=W^3$, then $(uW)^2$ is contained in $W^8$, whence $uW$ has doubling at most $k^7$. Since $|uW|/|V|=O(k^{O(1)})$, thus $V$ has doubling $O(k^{O(1)})$, and now by Theorem 3.3 $V$ contains a product-free subset of density $1/O(k^{O(1)})$, as needed.

Now fix $k\geqslant 1$, and suppose there are at least $|X|^2/k$ pairs $(x,y)\in X\times X$ with $xy\in X$. By the Fact 2.4, there are then subsets $Y,Z\subseteq X$, each of size $|X|/O(k^{O(1)})$, such that $|YZ|=O(k^{O(1)})|X|$. By Fact 2.3, there is then an $O(k^{O(1)})$-approximate subgroup $W\subseteq G$ of size at most $O(k^{O(1)})|X|$ and an element $u\in G$ such that $|Y\cap uW|$ has size at least $|X|/O(k^{O(1)})$; now one applies the previous paragraph to $V=Y\cap uW$ to obtain a product-free subset of $Y$ of size $|X|/O(k^{O(1)})$, giving the desired result.

So, we are able to handle the two edge cases, where either `very many' or `very few' pairs of elements of $X$ have product landing in $X$. However, it is entirely unclear to us how to handle the more general question, and we thus close our remarks here.

\appendix
\section{Product-free sets in solvable groups}
In \citep{alon_kleitman}, Alon and Kleitman showed that every finite set of non-zero elements of an abelian group contains a sum-free subset of density at least $2/7$. In this section we point out an analogous fact for solvable groups, which follows from a weighted version of Alon and Kleitman's theorem. The proof is a completely standard adaptation of their argument.

First we need the following easy fact, which is observed in \citep{alon_kleitman}; let $\alpha=1/4$.

\begin{fact}
    For every non-zero finite cyclic group $(H,+)$, there is a sum-free subset $I\subseteq H$ such that $|K\cap I|\geqslant \alpha|K|$ for every non-zero subgroup $K\leqslant H$. (Indeed when $H=\mathbb{Z}/n\mathbb{Z}$ one may take $I$ to be the image of $\{\lfloor n/3\rfloor+1,\dots,\lfloor 2n/3\rfloor\}$ in $H$.)
\end{fact}

Now the following is the weighted version of Alon and Kleitman's result, using an identical argument as theirs:

\begin{lemma}
    Let $(G,+,0)$ be an abelian group and $B\subseteq G\setminus\{0\}$ a finite subset, and for each $b\in B$ let $w_b$ be a positive integer. Then there is a sum-free subset $A\subseteq B$ with $\sum_{a\in A}w_a\geqslant \alpha \sum_{b\in B}w_b$.
\end{lemma}
\begin{proof}
    Every finite subset of an abelian group is Freiman isomorphic to a subset of $(\mathbb{Z}/n\mathbb{Z})^s$ for some $n,s$, so we may assume without loss of generality that $G=(\mathbb{Z}/n\mathbb{Z})^s$. Let $I\subseteq\mathbb{Z}/n\mathbb{Z}$ be given for $H=\mathbb{Z}/n\mathbb{Z}$ by Fact A.1.
    
    Now, for each $b=(b_1,\dots,b_s)\in B$, consider the homomorphism $f_b:G\to\mathbb{Z}/n\mathbb{Z}$ defined by $(c_1,\dots,c_s)\mapsto \sum_{i\in[s]}c_ib_i$. Note that, since $(0,\dots,0)\notin B$, $\operatorname{im}(f_b)$ is a non-zero subgroup of $\mathbb{Z}/n\mathbb{Z}$ for every $b\in B$. Now, pick $c\in G$ uniformly at random. Consider the random quantity $W_c=\sum_{b\in B}w_b1_{f_b(c)\in I}$. Since $f_b$ is a group homomorphism, its fibers have constant size, and so $\mathbb{P}[f_b(c)\in I]=|\operatorname{im}(f_b)\cap I|/|\operatorname{im}(f_b)|$ for each $b\in B$. Since also $\operatorname{im}(f_b)$ is non-zero, by Fact A.1 $\mathbb{P}[f_b(c)\in I]\geqslant\alpha$ for each $b\in B$, and so in particular $\mathbb{E}[W_c]\geqslant\alpha\sum_{b\in B}w_b$. 
    
    Thus we may find $c_0\in G$ with $W_{c_0}\geqslant\alpha\sum_{b\in B}w_b$. Now, let $A=\{a\in B:f_a(c_0)\in I\}$. Then $\sum_{a\in A}w_a=W_{c_0}\geqslant\alpha\sum_{b\in B}w_b$, and I claim that $A$ is sum-free. Indeed, suppose there are $a,a'\in A$ with $a+a'\in A$. Note that $f_a(c_0)+f_{a'}(c_0)=f_{a+a'}(c_0)$; by definition of $A$ all three of these quantities lie in $I$, contradicting that $I$ is sum-free.
\end{proof}

\begin{corollary}
    Let $(G,\cdot,1)$ be a solvable group admitting a length-$(n+1)$ subnormal series with abelian factor groups. Then every finite subset $C\subseteq G\setminus\{1\}$ contains a product-free subset of size at least $\alpha|C|/2^n$. 
\end{corollary}
\begin{proof}
    By induction on $n$. When $n=0$, $G$ is abelian, and so the base case is handled. For the inductive step, suppose $n>0$. Then there is a normal subgroup $H\triangleleft G$ such that $G/H$ is abelian and $H$ admits a length-$n$ subnormal series with abelian factor groups. Let $C\subseteq G$ be any finite subset. If $|C\cap H|\geqslant|C|/2$, then we are done by the inductive hypothesis. Thus suppose $|C\setminus H|\geqslant |C|/2$. Let $\pi:G\to G/H$ denote the projection map, and let $B=\pi(C\setminus H)$. For each $b\in B$, let $w_b=|\pi^{-1}(b)\cap C|$; by Lemma A.2, there is a product-free subset $A\subseteq B$ with $\sum_{a\in A}w_a\geqslant\alpha\sum_{b\in B}w_b$. But $\sum_{a\in A}w_a=|\pi^{-1}(A)\cap C|$, and $\sum_{b\in B}w_b=|C\setminus H|\geqslant |C|/2$. So $|\pi^{-1}(A)\cap C|\geqslant\alpha|C|/2$. The preimage of a product-free set under a group homomorphism is product-free, and so taking $\pi^{-1}(A)\cap C$ gives the desired product-free set.
\end{proof}

\newpage


\end{document}